\documentclass[11pt, a4paper]{amsart}

\usepackage{amsmath}
\usepackage{amsfonts}
\usepackage{amssymb}
\usepackage{amsthm}
\usepackage[active]{srcltx}

\numberwithin{figure}{section}

\newtheorem{theorem}{Theorem}[section]
\newtheorem{lemma}[theorem]{Lemma}

\theoremstyle{definition}
\newtheorem{definition}[theorem]{Definition}
\newtheorem{example}[theorem]{Example}

\numberwithin{equation}{section}

\def\kint_#1{\mathchoice%
          {\mathop{\kern 0.2em\vrule width 0.6em height 0.69678ex depth -0.58065ex
                  \kern -0.8em \intop}\nolimits_{\kern -0.4em#1}}%
          {\mathop{\kern 0.1em\vrule width 0.5em height 0.69678ex depth -0.60387ex
                  \kern -0.6em \intop}\nolimits_{#1}}%
          {\mathop{\kern 0.1em\vrule width 0.5em height 0.69678ex depth -0.60387ex
                  \kern -0.6em \intop}\nolimits_{#1}}%
          {\mathop{\kern 0.1em\vrule width 0.5em height 0.69678ex depth -0.60387ex
                  \kern -0.6em \intop}\nolimits_{#1}}}
\def\vintslides_#1{\mathchoice%
          {\mathop{\kern 0.1em\vrule width 0.5em height 0.697ex depth -0.581ex
                  \kern -0.6em \intop}\nolimits_{\kern -0.4em#1}}%
          {\mathop{\kern 0.1em\vrule width 0.3em height 0.697ex depth -0.604ex
                  \kern -0.4em \intop}\nolimits_{#1}}%
          {\mathop{\kern 0.1em\vrule width 0.3em height 0.697ex depth -0.604ex
                  \kern -0.4em \intop}\nolimits_{#1}}%
          {\mathop{\kern 0.1em\vrule width 0.3em height 0.697ex depth -0.604ex
                  \kern -0.4em \intop}\nolimits_{#1}}}

\newcommand{\R}{\mathbb{R}}

\renewcommand{\limsup}{\operatornamewithlimits{lim \, sup}}

\def\XXint#1#2#3{{\setbox0=\hbox{$#1{#2#3}{\int}$}
\vcenter{\hbox{$#2#3$}}\kern-.5\wd0}}

\DeclareMathOperator*{\loc}{loc}
\DeclareMathOperator*{\lip}{Lip}
\DeclareMathOperator*{\supp}{supp}

\title[]{Parabolic comparison principle and quasiminimizers in metric measure spaces}
\author{Juha Kinnunen and Mathias Masson}

\begin{document}
\begin{abstract}
We give several characterizations of parabolic (quasisuper)-mi\-nimizers in a metric measure space equipped with a doubling measure and supporting a Poincar\'e inequality. We also prove a version of comparison principle for super- and subminimizers on parabolic space-time cylinders and a uniqueness result for minimizers of a boundary value problem. We also give an example showing that the corresponding results do not hold, in general, for quasiminimizers even in the Euclidean case.
\end{abstract}

\subjclass[2010]{30L99, 35K92}
\maketitle

\section{Introduction}

This note studies properties of quasiminimimizers of parabolic variational inequalities of the $p$-Laplacian type in metric measure spaces. In the Euclidean case the prototype equation is 
\[
\frac{\partial u}{\partial t}-\textrm{div}(|Du|^{p-2}Du)= 0,
\qquad 1<p<\infty,
\]
and the corresponding variational inequality 
\[
p\iint u\frac{\partial\phi}{\partial t}\,dx\,dt
+\iint|Du|^p\,dx\,dt
\le\iint|Du+D\phi|^p\,dx\,dt
\]
for every compactly supported test function $\phi$. 
Roughly speaking, a quasiminimizer is a function which satisfies this condition up to multiplicative constants.
The precise definitions will be given later. Our results and methods also hold for more general variational
kernels of the $p$-Laplacian type, but as we shall see, in the class of quasiminimizers it seems to be enough
to consider the prototype case.

The main purpose of the study of quasiminimizers is to provide arguments in the calculus of variations that are only
based on energy estimates and, consequently, are applicable in more general contexts than Euclidean spaces. In particular, we do not always have the correspoding non-linear partial differential equation available. This is a challenging problem already in the Euclidean case with the Lebesgue measure and our results are relevant already in that case. In the elliptic case 
quasiminimizers were introduced by Giaquinta and Giusti in \cite{GG82} and \cite{GG84} and in the parabolic case by Wieser in \cite{W87}. See also \cite{Z93} and \cite{Z94}. First they were used a tool in the regularity theory, but later it has turned out that they have theory which is of independent interest, see \cite{B06}, \cite{KM03} and \cite{KS01}. 

Until recently, most of the results have been obtained for elliptic quasiminimizers. Recent papers \cite{KMMP12},
\cite{MMPP13} and \cite{MS11} extend the regularity theory of quasiminimizers to metric measure spaces with a doubling measure
and a Poincar\'e inequality. These are rather standard assumptions in analysis on metric measure spaces and
they are also assumed to hold throughout this note. We point out that there is a large literature on time independent variational problems in metric measure spaces, but  parabolic quasiminimizers open an opportunuty
to consider more general time dependent variational problems. So far there are only few references and many fundamental open problems remain. 

In this note, we shall focus on two aspects. First, we give several characterizations of parabolic quasisupermininimizers extending the results in \cite{KM03} and \cite{B06}.
We show, by an example, that quasiminimizers do not have comparison principle and a boundary value problem does not have a unique solution, in general.  However, if we restrict our attention to the minimizers, that are quasiminimizers with the constant one, then we have a parabolic comparison principle for super- and subminimizers and a uniqueness result for a boundary value problem for the minimizers. We prove these
results in the second part of the paper. It has been shown in \cite{W87} that the minimizers have a certain
amount of regularity in time. More precisely, the time derivative of a quasiminimizer belongs to the dual of the corresponding  parabolic Sobolev space. The corresponging result does not hold for super- and subminimizers
and a delicate smoothing argument is needed.
In the Euclidean case with the Lebesgue measure, the comparision principle is not only sufficient but also necessary condition for a function to be a superminimizer, but in the metric space the
theory for the parabolic obstacle problems is currently missing, see \cite{KL06}, \cite{KKP10} and \cite{KKS09}. Our uniqueness result extends results of \cite{W87} to the metric setting.

\section{preliminaries}
\subsection{Doubling measure.}
Let $(X,d)$ be a complete metric space.
A measure $\mu$ is said to be doubling if there exists a universal constant $C_\mu\ge 1$ such that
\[
 \mu(B(x,2r))\le C_\mu \mu(B(x,r)),
\]
for every $r>0$ and $x\in X$. Here $B(x,r)$ denotes the open ball with the center $x$ and radius $r$. We assume that the measure is nontrivial in the sense that $0<\mu(B(x,r))<\infty$ for every $x\in X$ and $r>0$.
Recall that a complete metric space with a doubling measure is proper, that is, every closed and bounded set is
compact.

\subsection{Upper gradients.} We assume that $\Omega$ is an open and bounded subset of $X$, although the boundedness is not always needed. Following \cite{HeinKosk98}, a non-negative Borel measurable function $g: \Omega \rightarrow [0, \infty]$ is said to be an upper gradient of a function $u: \Omega\rightarrow [-\infty, \infty]$ in $\Omega$, if for all compact rectifiable paths $\gamma$ joining $x$ and $y$ in $\Omega$ we have 
\begin{align}\label{upper gradient}
|u(x)-u(y)|\leq \int_{\gamma} g \,ds.
\end{align}
In case $u(x)=u(y)=\infty$ or $u(x)=u(y)=-\infty$, the left side is defined to be $\infty$. 
Throughout the paper we assume that $1<p<\infty$. A property is said to hold for $p$-almost all paths, if the set of non-constant paths for which the property fails is of zero $p$-modulus. Following \cite{Shan00}, if \eqref{upper gradient} holds for $p$-almost all paths $\gamma$ in $X$, then $g$ is said to be a $p$-weak upper gradient of $u$. 
This refinement is related to the the existence of a minimimal weak upper gradient and otherwise it does not play any other role in our study.

When $u$ is a measurable function and it has an upper gradient $g\in L^p(\Omega)$, it can be shown \cite{Shan01} that  there exists a minimal $p$-weak upper gradient of $u$, we denote it by $g_u$, in the sense that $g_u$ is a $p$-weak upper gradient of $u$
and for every $p$-weak upper gradient $g$ of $u$ it holds $g_u\leq g$ $\mu$-almost everywhere in $\Omega$. 
Moreover, if $v=u$ $\mu$-almost everywhere in a Borel set $E\subset \Omega$, then $g_v=g_u$ $\mu$-almost everywhere in $E$. 
For more on upper gradients in metric spaces we refer to  \cite{BB11} and \cite{Hein01}.

\subsection{Newtonian spaces.} Following \cite{Shan00}, for $u\in L^p(\Omega)$, we define
\begin{align*}
\|u\|_{1,p,\Omega}=\|u\|_{L^p(\Omega,\mu)}+\|g_u\|_{L^p(\Omega,\mu)},
\end{align*}
and $\widetilde N^{1.p}(\Omega)= \{ u\,:\, \|u\|_{1,p,\Omega}<\infty\}$.
An equivalence relation in $\widetilde N^{1,p}(\Omega)$ is defined by saying that $u\sim v$ if
$\|u-v\|_{\widetilde N^{1,p}(\Omega)}=0$.
The \emph{Newtonian space} $N^{1,p}(\Omega)$ is defined to be the space $\widetilde N^{1,p}(\Omega)/ \sim$, with the norm defined above.

A function $u$ belongs to the local Newtonian space $N_{\loc}^{1,p}(\Omega)$ if it belongs to $N^{1,p}(\Omega')$ for every $\Omega' \Subset \Omega$. 
The notation $\Omega' \Subset \Omega$ means that $\overline{\Omega'}$ is a compact subset of $\Omega$.
The Newtonian space with zero boundary values is defined as 
\[
N_0^{1,p}(\Omega)=\{u|_\Omega: u\in N^{1,p}(X),u=0\text{ in }X\setminus\Omega\}.
\]
In practice, this means that a function belongs to $N_0^{1,p}(\Omega)$ if and only if its zero extension to $X\setminus\Omega$ belongs to $N^{1,p}(X)$. 
For more properties of Newtonian spaces, see \cite{Shan00, KKM00, BB11, Hein01}.

\subsection{Poincar\'e's inequality}
The space $X$ is said to support a weak $(1,p)$-Poincar\'e inequality with $1\leq p < \infty$, if there exist constants $P_0>0$ and $\tau\ge 1$ such that
\[
\frac{1}{\mu(B(x,r))}\int_{B(x,r)}|u-u_{B(x,r)}|\,d\mu
\le P_0 r\left(\frac{1}{\mu(B(x,\tau r))}\int_{B(x,\tau r)} g_u^p \, d\mu\right)^{1/p},
\]
for every $u\in N^{1,p}(X)$, $x\in X$ and $r>0$.
If the measure is doubling and the space supports a $(1,p)$-Poincar\'e inequality, then Lipschitz continuous
functions are dense in the Newtonian space. This will be useful for us, since the test functions in the
definition of quasiminimizers are assumed to be Lipschitz continuous.

\subsection{General assumptions}
Throughout this paper we assume $(X,d,\mu)$ to be a complete metric space, equipped with a positive doubling Borel measure $\mu$ which  supports a weak $(1,p)$-Poincar\'e inequality for some $1\leq p<\infty$. 

\subsection{Parabolic upper gradients and Newtonian spaces}
 From now on we will denote the product measure by $d\nu=d\mu\, dt$. Whenever $t$ is such that $u(\cdot,t)\in N^{1,p}(\Omega)$, we define the parabolic minimal $p$-weak upper gradient of $u$ in a natural way by setting
\[
g_u(x,t)=g_{u(\cdot,t)}(x), 
\]
at $\nu$-almost every $(x,t)\in \Omega\times (0,T)=\Omega_T$ and we call the parabolic minimal $p$-weak upper gradient as the upper gradient. 

We define the parabolic Newtonian space  $L^p(0,T;N^{1,p}(\Omega))$ to be the space of functions $u(x,t)$ such that for almost every $0<t<T$ the function $u(\cdot, t)$ belongs to $N^{1,p}(\Omega)$, and
\[
\int\limits_{0}^{T}\|u(\cdot,t)\|_{1,p,\Omega}^p\, dt < \infty.
\]
We say that $u\in L_{\loc}^p(0,T;N_{\loc}^{1,p}(\Omega))$ if for every $0<t_1<t_2<T$ and $\Omega'\Subset\Omega$ we have $u\in L^p(t_1,t_2;N^{1,p}(\Omega'))$.
Finally, we say that $u\in L_{\textrm{c}}^p(0,T;N^{1,p}(\Omega))$ if for some $0<t_1<t_2<T$, we have $u(\cdot,t)=0$ outside $[t_1,t_2]$.

The next Lemma on taking limits of upper gradients will also be needed later in this paper.
\begin{lemma}\label{convergence of upper gradient} Let $u$ be such that $g_u\in L_{\loc}^p(\Omega_T)$. Then the following statements hold:
\begin{itemize} \item[(a)] As $s\rightarrow 0$, we have $g_{u(x,t-s)-u(x,t)}\rightarrow 0$ in $L_{\loc}^p(\Omega_T)$.
\item[(b)] As $\varepsilon\rightarrow 0$, we have $g_{u_\varepsilon -u} \rightarrow 0$ pointwise $\nu$-almost everywhere in $\Omega_T$ and in $L_{\loc}^p(\Omega_T)$.
\end{itemize}
\begin{proof}
See Lemma 6.8 in \cite{MS11}. 
\end{proof}

\end{lemma}

\section{Parabolic quasiminimizers}

In this section we define parabolic quasiminimizers in metric measure spaces and give several characterizations for them. We begin with a brief discussion about the Euclidean case.

\subsection{Euclidean case} 
There is a variational approach to the $p$-parabolic equation
\begin{equation}\label{ppe} 
\frac{\partial u}{\partial t}-\textrm{div}(|Du|^{p-2}Du)= 0,
\qquad 1<p<\infty.
\end{equation}
To see this, first assume that $u\in L^p_{\loc}(0,T;W^{1,p}_{\loc}(\Omega))\cap L^2(\Omega_T)$ 
is a weak solution of \eqref{ppe} and denote $dz=dx\,dt$.
Let $U\Subset\Omega_T$ and $\phi\in C^\infty_0(U)$. 
Then
\[
\begin{split}
\int_U|Du|^p\,dz
&=\int_U|Du|^{p-2}Du\cdot Du\,dz
\\
&=\int_U|Du|^{p-2}Du\cdot (Du+D\phi)\,dz
-\int_Uu\frac{\partial\phi}{\partial t}\,dz.
\end{split}
\]
from which it follows that
\begin{equation}\label{qme}
p\int_Uu\frac{\partial\phi}{\partial t}\,dz
+\int_U|Du|^p\,dz
\le\int_U|Du+D\phi|^p\,dz.
\end{equation}
In the last step we used Young's inequality and rearranged terms.

Assume then that  \eqref{qme} holds for every $U\Subset\Omega_T$ and let $\phi\in C^\infty_0(\Omega_T)$. 
Then $\varepsilon\varphi\in C^\infty_0(U)$ with $\varepsilon>0$ for some $U\Subset\Omega_T$
and we have
\[
\varepsilon p\int_Uu\frac{\partial\phi}{\partial t}\,dz+\int_U|Du|^p\,dz
\le\int_U|Du+\varepsilon D\phi|^p\,dz
\]
from which we conclude that
\[
p\int_Uu\frac{\partial\phi}{\partial t}\,dz
+\int_U\frac1\varepsilon(|Du|^p-|Du+\varepsilon D\phi|^p)\,dz
\le0.
\]
As $\varepsilon\to0$, by the dominated convergence theorem we arrive at
\[
\int_Uu\frac{\partial\phi}{\partial t}\,dz
-\int_U|Du|^{p-2}Du\cdot D\phi\,dz
\le0.
\]
The reverse inequality follows by choosing $-\varepsilon\phi$ as the test function
and consequently $u$ is a weak solution of \eqref{ppe}.
If $u\in L^p_{\loc}(0,T;W^{1,p}_{\loc}(\Omega))\cap L^2_{\loc}(\Omega_T)$ satisfies \eqref{qme} for every $U\Subset\Omega_T$, we say that $u$ is a parabolic minimizer.
Hence every weak solution of the $p$-parabolic equation is a parabolic minimizer and, conversely, every
parabolic minimizer is a weak solution of the $p$-parabolic equation. 

Let us then consider more general equations 
\begin{equation}\label{ape}
\frac{\partial u}{\partial t}-\textrm{div}A(x,t,Du)= 0,
\qquad 1<p<\infty,
\end{equation}
of the the $p$-Laplacian type,
where $A: \R^n\times\R\times \R^n \to \R $ satisfies the following assumptions:
\begin{enumerate}
\item $ (x,t)\mapsto A(x,t,\xi)$ is measurable for every $\xi$,
\item $\xi \mapsto A(x,t,\xi)$ is continuous for almost every $(x,t)$, and
\item  there exist $0<c_1\leq c_2<\infty$ such that for every $\xi$ and almost every $(x,t)$, we have
\[
A(x,t,\xi)\cdot\xi\ge c_1|\xi|^p
\quad\text{and}\quad
|A(x,t,\xi)|\le c_2|\xi|^{p-1}.
\]
\end{enumerate}

As above, we can show that a weak solution of \eqref{ape} satisfies
\[
\int_Uu\frac{\partial\phi}{\partial t}\,dz
+c_1\int_U|Du|^p\,dz
\le c_2\int_U|Du|^{p-1}|Du+D\phi|\,dz
\]
for every $U\Subset\Omega_T$ and $\phi\in C^\infty_0(U)$. Young's
inequality implies that
there are constants $\alpha>0$ and $K\ge1$,
depending only on $p$, $c_1$ and $c_2$, such that
\[
\alpha\int_Uu\frac{\partial\phi}{\partial t}\,dz
+\int_U|Du|^p\,dz
\le K\int_U|Du|^p\,dz
\]
for every $U\Subset\Omega_T$ and $\phi\in C^\infty_0(U)$.
A function $u\in L^p_{\loc}(0,T;W^{1,p}_{\loc}(\Omega))\cap L^2_{\loc}(\Omega_T)$, which
satisfies this property is called a parabolic quasiminimizer. 
In contrast with the elliptic case, two parameters $\alpha$ and $K$ are required in the parabolic
case to be able to obtain a notion of a quasiminimizer which includes the whole
class of equations of the type \eqref{ape}. 
Another possibility would be to study qasiminimizers with more general variational integrals than
the $p$-Dirichlet integral, but we leave this for the interested reader.

\begin{example}
When $K>1$, then being a quasiminimizer is not a local
property. Indeed, consider the function $u:(0,\infty)\times(0,1)\to\R$ defined by setting
\[
u(x,t)=\frac{x-(i-1)}k+\sum_{j=1}^{i-1}\frac1j,
\]
when $i-1<x\le i$ with $i=1,2,\dots$ 
This function is an elliptic quasiminimizer with some $K>1$ when tested on intervals
of finite length by a criterion given in \cite{GG84}. 
However, it fails to be a quasiminimizer on the whole positive axis by the same criterion. 
\end{example}

\subsection{Metric case}
The advantage of the notion of a quasiminimizer is that it makes sense
also in metric spaces and this enables us to develop the theory of nonlinear
parabolic partial differential equations also in the metric context.

\begin{definition}
Suppose that $\Omega\subset X$ is an open set and $0<T<\infty$. Let $1<p<\infty$, $\alpha>0$ and  $K\geq 1$. We say that a function $u\in L_{\loc}^p(0,T;N_{\loc}^{1,p}(\Omega))\cap L_{\loc}^2(\Omega_T)$ is a parabolic $K$-quasiminimizer, if for every open $U \Subset \Omega_T$ and for all functions $\phi \in\lip(\Omega_T)$ such that $\{\phi\neq 0\}\subset U$, we have
\begin{align}\label{quasiminimizer definition}
\alpha\int_Uu \frac{\partial\phi}{\partial t}\,d\nu+\int_U g_u^p\,d\nu \leq K\int_U g_{u+\phi}^p\,d\nu.
\end{align}
A function $u\in L_{\loc}^p(0,T;N_{\loc}^{1,p}(\Omega))\cap L_{\loc}^2(\Omega_T)$ is a parabolic $K$-quasisuper-minimizer, if \eqref{quasiminimizer definition} holds for every open $U\Subset \Omega_T$ and for all nonnegative functions $\phi \in \lip(\Omega_T) $ such that $\{\phi>0\}\subset U$. A function $u$ is a parabolic $K$-quasisubminimizer, if  $-u$ is a parabolic $K$-quasisuperminimizer. 
If $K=1$, then the parabolic $K$-quasiminimizer and parabolic $K$-quasi-superminimizer are called parabolic minimizer and parabolic superminimizer, respectively.
\end{definition}

We proceed by proving characterizations for parabolic quasiminimizers. The characterizations are proved for parabolic quasisuperminimizers. However, the reader should note that Lemmas \ref{characterization 1} through \ref{quasiminimizer after partial integration} can be formulated and proved also for $K$-quasisubminimizers, after replacing the word \lq nonnegative\rq\ in the proofs with the word \lq nonpositive\rq.
In particular, this implies that the corresponding results also hold for parabolic minimizers.

\begin{lemma}\label{characterization 1}
A function $u\in L_{\loc}^p(0,T;N_{\loc}^{1,p}(\Omega))\cap L_{\loc}^2(\Omega_T)$ is a parabolic $K$-quasisuperminimizer if and only if for  every  $\nu$-measurable set $E \Subset \Omega_T$ and every  nonnegative $\phi \in\lip(\Omega_T)$ such that $\{\phi>0\}\subset E$,  we have
\begin{align*}
\alpha \int_E u \frac{\partial\phi}{\partial t}\,d\nu+\int_E g_u^p\,d\nu \leq K \int_E g_{u+\phi}^p\,d\nu.
\end{align*}

\begin{proof}
Assume that $u$ is a parabolic $K$-quasisuperminimizer in $\Omega_T$ and let ${E}\Subset \Omega$ be a $\nu$-measurable set.  Let $\phi \in \lip(\Omega)$ be nonnegative with the property $\{\phi>0\}\subset E$. Since $E$ is $\nu$-measurable, and since $u+\phi \in L_{\loc}^p(0,T;N_{\loc}^{1,p}(\Omega))\cap L_{\loc}^2(\Omega_T)$, by the regularity of $\mu$ there exists an open set $U\Subset \Omega_T$ such that 
\begin{align*}
\int_{U\setminus E} g_{u+\phi}^p\, d\nu < \frac{\varepsilon}{K}.
\end{align*}
Moreover, $\{\phi\neq 0\} \subset E$, and since ${\phi}$ is continuous with respect to time, we have 
\[
\nu((F\setminus E)\cap\{\partial \phi/\partial t\neq 0\})=0. 
\]
Since $u$ is a $K$-quasisuperminimizer, we arrive at 
\begin{align*}
\alpha\int_{E}& u \frac{\partial\phi}{\partial t}\,d\nu+\int_{{E}} g_u^p\,d\nu
\leq \alpha\int_{U} u \frac{\partial \phi}{\partial t}\, d\nu+\int_{U} g_u^p\,d\nu\\
&\leq K \int_{E} g_{u+\phi}^p\,d\nu+ K \int_{U\setminus E} g_{u+ \phi}^p\,d\nu
\leq K \int_{E} g_{u+\phi}^p\,d\nu+ \varepsilon.
\end{align*} 
This holds for every $\varepsilon$, and so one direction of the claim is true by passing $\varepsilon\to0$. 
The other direction is immediate, since open sets are $\nu$-measurable.
\end{proof}
\end{lemma}

%

The following two characterizations are often useful in applications and in many cases can be taken
as the definition of a parabolic $K$-quasiminimizer.

\begin{lemma}\label{characterization 2}
A function $u\in L_{\loc}^p(0,T;N_{\loc}^{1,p}(\Omega))\cap L_{\loc}^2(\Omega_T)$ is a $K$-quasi-superminimizer if and only if  for every nonnegative  $\phi \in \lip(\Omega_T)$   such that $\{\phi \neq 0\}\Subset \Omega_T$, we have
\begin{align}\label{quasiminimizer on positive part}
\alpha\int_{\{\phi\neq 0\}} u \frac{\partial\phi}{\partial t}\,d\nu+\int_{\{\phi\neq 0\}} g_u^p\,d\nu \leq K \int_{\{\phi\neq 0\}} g_{u+\phi}^p\,d\nu.
\end{align}

\begin{proof}
Suppose first that \eqref{quasiminimizer on positive part} holds. Let $\phi \in \lip(\Omega_T)$ be nonnegative with $\{\phi\neq 0\}\Subset \Omega_T$. Let $U \Subset \Omega_T$ be an open set such that $\{\phi\neq 0\}\subset U$.  Since $\phi$ is continuous with respect to time, we have 
\[
\nu(\{\phi =0, \, \partial \phi/\partial t \neq 0\})=0,
\] 
and consequently
\begin{align*}
\alpha\int_{U}& u\frac{ \partial \phi}{\partial t} \,d\nu+\int_{U} g_u^p \, d\nu
= \alpha\int_{\{\phi \neq 0\}} u\frac{ \partial \phi}{\partial t} \,d\nu+\int_{ \{\phi \neq 0\}} g_u^p \, d\nu +\int_{U\cap \{\phi =0\}} g_u^p \, d\nu\\
 &=K\int_{ \{\phi \neq 0\}} g_{u+\phi}^p \, d\nu+ \int_{U\cap \{\phi =0\}} g_u^p \, d\nu
 \leq K\int_{U} g_{u+\phi}^p \, d\nu.
\end{align*}
This shows that $u$ is a parabolic $K$-quasisuperminimizer.

Suppose then that $u$ is a $K$-quasisuperminimizer. Consider an open set $U\Subset \Omega_T$ and a nonnegative function $\phi \in \lip(\Omega_T)$ such that $\{\phi\neq 0\} \subset U$. The set $\{\phi\neq 0\}$
is $\nu$-measurable and $\{\phi\neq 0\}\Subset \Omega_T$,  and so by Lemma \ref{characterization 1} we have \eqref{quasiminimizer on positive part}. This completes the proof.
\end{proof}

\end{lemma}

\begin{lemma} A function $u\in L_{\loc}^p(0,T;N_{\loc}^{1,p}(\Omega))\cap L_{\loc}^2(\Omega_T)$ is a $K$-quasi-superminimizer if and only if  for every nonnegative  $\phi \in \lip(\Omega_T)$   such that $\supp\phi\subset \Omega_T$, we have
\begin{align}\label{quasiminimizer on positive part 2}
\alpha\int_{\supp\phi} u \frac{\partial\phi}{\partial t}\,d\nu+\int_{\supp\phi} g_u^p\,d\nu \leq K \int_{\supp\phi} g_{u+\phi}^p\,d\nu.
\end{align}

\begin{proof}
Suppose $u$ is a $K$-quasisuperminimizer and let $\phi \in \lip(\Omega_T)$ be nonnegative such that $\supp\phi \subset \Omega_T$. Because $\phi$ is continuous, we have 
\[
\nu(\{\phi =0, \partial \phi/\partial t \neq 0\})=0.
\] 
Since $K\geq 1$, we can write 
\begin{align*}
\int_{\supp\phi}&u\frac{\partial \phi}{\partial t} \, d\nu +  \int_{\supp\phi} g_{u}^p \, d\nu\\
&\le \int_{\{\phi \neq 0\}}u\frac{\partial \phi}{\partial t} \, d\nu +  \int_{\{\phi \neq 0 \}} g_{u}^p \, d\nu
+ \int_{\supp\phi\setminus \{\phi \neq 0 \}} g_{u}^p \, d\nu\\
& \leq K\int_{\{\phi \neq 0 \}} g_{u+\phi}^p \, d\nu+ \int_{\supp\phi\setminus \{\phi \neq 0 \}} g_{u+\phi}^p \, d\nu\\
&\leq K\int_{\supp\phi} g_{u+\phi}^p \, d\nu.
\end{align*}

Let then $\phi\in \lip(\Omega_T)$ be such that $\{\phi\neq 0\}\Subset \Omega_T$. For $i=1,2,\dots$, define
\begin{align*}
\psi_i=(\phi-i^{-1})_+-(\phi+i^{-1})_-.
\end{align*}
Then for each $i$ we have 
\[
\psi_i\in \lip(\Omega_T), 
\quad
\{\psi_i\neq 0\}\subset \{\phi\neq 0\}, 
\quad
\nu(\{\phi\neq 0\}\setminus \{\psi_i\neq 0\})\rightarrow 0
\]
as $i \rightarrow \infty$, and
\begin{align*}
\frac{\partial \psi_i}{\partial t}=\frac{\partial \phi}{\partial t}
\quad\text{and}\quad g_{\psi_i}=g_{\phi},  
\end{align*}
in the set $\{\psi_i\neq 0\}$.
Moreover, since $\phi$ is continuous, we have supp$\,\psi_i\subset \{|\phi|\geq i^{-1}\}$. Let $\varepsilon>0$. By the absolute continuity of the integral and the above properties, there exists a large enough $i$ such that
\begin{align*}
\int_{\{\phi \neq 0\}}&u\frac{\partial \phi}{\partial t} \, d\nu + C_1 \int_{\{\phi \neq 0 \}} g_{u}^p \, d\nu
\leq\int_{\textrm{supp}\, \psi_i}u\frac{\partial \phi}{\partial t} \, d\nu + C_1 \int_{\supp\psi_i} g_{u}^p \, d\nu+\varepsilon\\&\le K  \int_{\supp\psi_i} g_{u-\psi_i}^p \, d\nu+\varepsilon
\leq K  \int_{\{\phi \neq 0 \}} g_{u-\phi}^p \, d\nu+2\varepsilon.
\end{align*}
Since this is true for any positive $\varepsilon$, by Lemma \ref{characterization 2} $u$ is a parabolic $K$-quasisuperminimizer. 
\end{proof}
\end{lemma}

It turns out that after mollifying  a $K$-quasisuperminimizer in time, we obtain estimates for test functions which do not necessarily have to be smooth in time.
In what follows $(\cdot)_\varepsilon$ denotes the standard time mollification
\begin{align*}
f_\varepsilon(x,t)=\int_{-\varepsilon}^{\varepsilon} \eta_\varepsilon(s) f(x,t-s)\,ds.
\end{align*}

\begin{lemma}\label{quasiminimizer after partial integration} Let  $u\in L_{\loc}^p(0,T;N_{\loc}^{1,p}(\Omega))\cap L_{\loc}^2(\Omega_T)$ be a parabolic $K$-quasi-superminimizer. Then for every nonnegative $\phi \in L^p(0,T;N^{1,p}(\Omega))\cap L^2(\Omega_T)$ such that $\{\phi\neq 0\}\Subset \Omega_T$, we have 
\begin{align}\label{inequality after partial integration}
-\alpha \int_{\{\phi \neq 0\}}\frac{\partial u_\varepsilon}{\partial t} \phi \,d\nu+ \int_{\{\phi \neq 0\}}(g_u^p)_\varepsilon\,d\nu \leq K \int_{\{\phi \neq 0\}}(g_{u(x,t-s)+\phi}^p)_\varepsilon\,d\nu
\end{align}
for every small enough positive $\varepsilon$.
Moreover, if $u\in L_{\loc}^p(0,T;N^{1,p}(\Omega))\cap L_{\loc}^2(\Omega_T)$, then the same inequality also holds for every nonnegative $\phi \in L_\textrm{c}^p(0,T; N_0^{1,p}(\Omega))\cap L^2(\Omega_T)$.

\begin{proof}
See Lemma 2.3, Corollary 2.4 and Lemma 2.7 in \cite{MMPP13}.
\end{proof}
\end{lemma}

\section{Parabolic comparison principle}

In this section we prove a comparison principle for minimizers, and as a consequence obtain the uniqueness of parabolic minimizers. We emphasize that it is essential to have $K=1$ in the discussion below. Indeed, we give
an example which shows that the
comparison principle and the uniqueness result do not hold for parabolic $K$-quasiminimizers when $K>1$. 

\begin{theorem}Let $u\in L^p(0,T;N^{1,p}(\Omega))\cap L^2(\Omega_T)$ be a parabolic superminimizer and let $v\in L^p(0,T;N^{1,p}(\Omega))\cap L^2(\Omega_T)$ be a parabolic subminimizer, both with the same constant $\alpha>0$ in \eqref{quasiminimizer definition}. Suppose $u\geq v$ near the parabolic boundary of $\Omega_T$, in the sense that for almost every $0<t<T$ we have the lateral boundary condition $(v(x,t)-u(x,t))_+\in N_0^{1,p}(\Omega)$ and also the initial condition
\begin{align}\label{initial condition}
\frac{1}{h}\int_0^h \int_{\Omega} (v-u)_+^2\,d\nu \rightarrow 0, \qquad \textrm{ as }h\rightarrow 0.
\end{align}
Then $u\geq v$ $\nu$-almost everywhere in $\Omega_T$.
\end{theorem}

\begin{proof}
Assume that $u$ and $v$ are as in the formulation of the result. Let $t'\in (0,T)$, and for $h>0$ define
\begin{align*}
\chi_h=\begin{cases}
 \frac{t-h}{h}, &h\leq t \leq 2h, \\
 1,& 2h \leq t \leq t'-2h,\\
 \frac{t'+h-t}{2h},&t'-h\leq t \leq t'+h,\\
 0,&\textrm{otherwise}.
\end{cases}
\end{align*}
Choose the test function $\phi=(v_\varepsilon-u_\varepsilon)_+\chi_{h}$.
By the assumptions made on $(u-v)_+$ near the lateral boundary of $\Omega_T$, we see that for $\varepsilon$ and $h$ small enough $\phi\in L_c^p(0,T;N^{1,p}_0(\Omega))\cap L^2(\Omega_T)$. 
Therefore, since $u$ is a superminimizer, by Lemma \ref{quasiminimizer after partial integration} we have
\begin{align*}
-\alpha\int_{\{\phi \neq 0\}}&\frac{\partial u_\varepsilon}{\partial t} (v_\varepsilon-u_\varepsilon)_+\chi_h\, d\nu+ \int_{\{\phi \neq 0\}}(g_u^p)_\varepsilon\,d\nu\\
&\leq  \int_{\{\phi \neq 0\}}(g_{u(x,t-s)+(v_\varepsilon-u_\varepsilon)\chi_h}^p)_\varepsilon\,d\nu.
\end{align*}
On the other hand, since $v$ is a subminimizer, we may use $-\phi$ as a test function, to obtain
\begin{align*}
\alpha\int_{\{\phi \neq 0\}} &\frac{\partial v_\varepsilon}{\partial t} (v_\varepsilon-u_\varepsilon)_+\chi_h\, d\nu+ \int_{\{\phi \neq 0\}}(g_v^p)_\varepsilon\,d\nu\\ 
&\leq  \int_{\{\phi \neq 0\}}(g_{v(x,t-s)+(u_\varepsilon-v_\varepsilon)\chi_h}^p)_\varepsilon\,d\nu.
\end{align*}
By adding the above two inequalities together we have
\begin{equation}\label{inequality before taking epsilon to zero}
\begin{split}
\alpha\int_{\{\phi\neq 0\}}\frac{1}{2} \frac{\partial}{\partial t} \left((v_\varepsilon-u_\varepsilon)_+^2\right)\chi_h \, d\nu+ \int_{\{\phi\neq 0\}} (g_u^p)_\varepsilon\, d\nu +\int_{\{\phi\neq 0\}} (g_v^p)_\varepsilon\, d\nu\\
\leq \int_{\{\phi \neq 0\}}(g_{u(x,t-s)+(v_\varepsilon-u_\varepsilon)\chi_h}^p)_\varepsilon\,d\nu+\int_{\{\phi \neq 0\}}(g_{v(x,t-s)-(v_\varepsilon+u_\varepsilon)\chi_h}^p)_\varepsilon\,d\nu.
\end{split}
\end{equation}
Next we note that by adding and subtracting, and then using Minkowski's inequality, and since $\chi_h$ does not depend on the spatial variable, we have
\begin{align*}
&\left(\int_{\{\phi \neq 0\}}g_{u(x,t-s)+(v_\varepsilon-u_\varepsilon)\chi_h}^p\,d\nu\right)_\varepsilon
\leq \left( \left( \left(\int_{\{\phi \neq 0\}}g_{v}^p\,d\nu\right)^{1/p}\right.\right.\\
&+ \left(\int_{\{\phi \neq 0\}}g_{u(x,t-s)-u}^p\,d\nu\right)^{1/p}+ \left(\int_{\{\phi \neq 0\}}g_{v_\varepsilon-v}^p\,d\nu\right)^{1/p}\\
&\left.\left.+\left(\int_{\{\phi \neq 0\}}g_{u-u_\varepsilon}^p\,d\nu\right)^{1/p}+\left(\int_{\{\phi \neq 0\}}g_{(v_\varepsilon-u_\varepsilon)}^p(\chi_h-1)^p\,d\nu\right)^{1/p}\right)^p\right)_\varepsilon.
\end{align*}
Lemma \ref{convergence of upper gradient} implies that
 \begin{align*}
 \limsup_{\varepsilon \rightarrow 0} \int_{\{\phi \neq 0\}}(g_{u(x,t-s)+v_\varepsilon-u_\varepsilon}^p)_\varepsilon\,d\nu \leq\int_{\{\varphi \neq 0\}}g_{v}^p\,d\nu,
 \end{align*}
 where we have denoted $\varphi=(v-u)_+\chi_{h}$. Similarly, we see that
 \begin{align*}
 \limsup_{\varepsilon \rightarrow 0} \int_{\{\phi \neq 0\}}(g_{v(x,t-s)-v_\varepsilon+u_\varepsilon}^p)_\varepsilon\,d\nu \leq\int_{\{\varphi \neq 0\}}g_{u}^p\,d\nu.
 \end{align*}
On the left hand side of  \eqref{inequality before taking epsilon to zero} we integrate by parts and use the definition of $\chi_h$ to obtain
\begin{align*}
\alpha\int_{\{\phi\neq 0\}}&\frac{1}{2} \left(\frac{\partial}{\partial t} (v_\varepsilon-u_\varepsilon)_+^2\right)\chi_h \, d\nu\\
&=\frac{\alpha}{4h}\int_{t'-h}^{t'+h}\int_{\Omega}(v_\varepsilon-u_\varepsilon)_+^2\,d\mu\,dt-\frac{\alpha}{2h}\int_{h}^{2h}\int_{\Omega}(v_\varepsilon-u_\varepsilon)_+^2\,d\mu\,dt
\end{align*}
Hence from \eqref{inequality before taking epsilon to zero} we obtain, after first taking the limit $\varepsilon\rightarrow 0$ and then $h\rightarrow 0$, and also taking into account  initial condition \eqref{initial condition},
\begin{align*}
\frac{\alpha}{4h}\int_{\Omega}&(v(x,t')-u(x,t'))_+^2\,d\mu +\int_{\{\varphi\neq 0\}} g_u^p\, d\nu +\int_{\{\varphi\neq 0\}} g_v^p\, d\nu\\
&\quad\leq \int_{\{\varphi\neq 0\}} g_v^p\, d\nu+ \int_{\{\varphi\neq 0\}} g_u^p\, d\nu
\end{align*}
for almost every $t'\in (0,T)$.
Since the upper gradient terms cancel each other, the above implies that for almost every $t'\in (0,T)$ we have 
\[
(v(x,t')-u(x,t'))_+=0
\]
at $\mu$-almost every $x\in \Omega$. This completes the proof.
\end{proof}
A parabolic minimizer is both a sub- and superminimizer, and so the comparison principle above immediately implies the following uniqueness result.

\begin{theorem}Let $u,v\in L^p(0,T;N^{1,p}(\Omega))\cap L^2(\Omega_T)$ be parabolic minimizers, both with the same constant $\alpha>0$ in \eqref{quasiminimizer definition}. 
Suppose that for almost every $t\in (0,T)$ we have $u-v \in N_0^{1,p}(\Omega)$, and suppose we have the initial condition  
\begin{align*}
\frac{1}{h}\int_0^h \int_{\Omega'} |v-u|^2\,d\nu \rightarrow 0, \qquad \textrm{ as }h\rightarrow 0.
\end{align*}
Then $u=v$ $\,\nu$-almost everywhere in $\Omega_T$.
\end{theorem}

\begin{example}
Consider the functions $u$ and $v$ which are solutions of the one-dimensional heat equations
\[
\frac{\partial u}{\partial t}-u^{''}= 0
\qquad
\text{and}
\qquad 
\frac{\partial v}{\partial t}-av^{''}= 0,\qquad a>1,
\]
in the domain $(0,1)\times(0,1)$ with zero boundary values on the lateral boundary 
and the same nontrivial initial values. Denote $dz=dx\,dt$.
By the separation of variables we see that $u\ne v$, but they have the
same boundary values on the parabolic boundary. 

As above, we can show that $v$ satisfies
\[
\begin{split}
\int_U&v\frac{\partial\phi}{\partial t}\,dz
+a\int_U|v'|^2\,dz
\le a\int_U|v'||v'+\phi'|\,dz
\\
&\le\Big(a-\frac12\Big)\int_U|v'|^2\,dz
+\frac12\frac{a^2}{2a-1}\int_U|v'+\phi'|^2\,dz
\end{split}
\]
for every $U\Subset(0,1)\times(0,1)$ and $\phi\in C^\infty_0(U)$. 
This implies that
\[
2\int_Uv\frac{\partial\phi}{\partial t}\,dz
+\int_U|v'|^2\,dz
\le\frac{a^2}{2a-1}\int_U|v'+\phi'|^2\,dz
\]
for every $U\Subset(0,1)\times(0,1)$ and $\phi\in C^\infty_0(U)$ and hence
$v$ is a parabolic quasiminimizer with 
\[
K=\frac{a^2}{2a-1}>1.
\]
In the other hand,  the function $u$ satisfies
\[
2\int_Uu\frac{\partial\phi}{\partial t}\,dz
+\int_U|u'|^2\,dz
\le\int_U|u'+\phi'|^2\,dz
\le K\int_U|u'+\phi'|^2\,dz
\]
for every $U\Subset(0,1)\times(0,1)$ and $\phi\in C^\infty_0(U)$ and hence
$u$ is a parabolic quasiminimizer with the
same constant as $v$ and, consequently, both uniqueness and the comparison principle do not hold
for quasiminimizers, in general.
\end{example}

\def\cprime{$'$} \def\cprime{$'$}

\vspace{0.3cm}
\noindent
\small{\textsc{J.K.} and M.M.,}
\small{\textsc{Department of Mathematics},}
\small{\textsc{P.O. Box 11100},}
\small{\textsc{FI-00076 Aalto University},}
\small{\textsc{Finland}}\\
\footnotesize{\texttt{juha.k.kinnunen@aalto.fi, mathiasmasson@hotmail.com}}

\end{document}